\documentclass[12pt,a4paper,reqno]{amsart}
\usepackage{amsfonts}
\usepackage{amssymb}
\usepackage{amsthm}
\usepackage{amsmath}
\usepackage{newlfont}
\usepackage{graphicx}
\usepackage{color}

\def\r{\mathbb R}
 
\def\s{\mathbb S}

\def\n{\mathbf n}

\newtheorem{theorem}{Theorem}[section]

\newtheorem{example}[theorem]{Example}

\title[Ruled surfaces that are critical points of the Dirichlet energy]{Classification of the ruled surfaces that are critical points of the Dirichlet energy}

\author[Rafael L\'opez]{Rafael L\'opez}
\address{Departamento de Geometr\'{\i}a y Topolog\'{\i}a\\
 Universidad de Granada 18071 Granada, Spain}
\email{rcamino@ugr.es}

\date{}

\begin{document}
\begin{abstract}
We classify all ruled surfaces in Euclidean space that are critical points of the Dirichlet energy, obtaining explicit parametrizations of these surfaces. \end{abstract}

\maketitle

\section{Introduction and statement of the results} \label{intro}

Let $\Sigma $ be a surface in Euclidean space $\r^3$ given as the graph $z=u(x,y)$ of a smooth function $u$ defined in a bounded domain   $\Omega\subset\r^2$. The Dirichlet energy of $u$ is the integral $\int_\Omega|Du|^2$. The function $u$ is a critical point of the Dirichlet energy for all   volume-preserving variations of $\Sigma$ if and only if $u$   is a critical point of the functional 
$$E[u]=\int_\Omega|Du|^2+\Lambda\int_\Omega u,$$
where  $\Lambda\in\r$ stands for a Lagrange multiplier. Notice that the integral $\int_\Omega u$ is the volume enclosed by $\Sigma$ with the solid cylinder $\Omega\times\r$. The Euler-Lagrange equation of $E$ is 
\begin{equation}\label{hh}
u_{xx}+u_{yy}=\frac{\Lambda}{2},
\end{equation}
 where $u=u(x,y)$. In the particular case $\Lambda=0$,  the function $u$ is harmonic. A surface $\Sigma$ satisfying locally \eqref{hh} is  called a {\it stationary surface} of the Dirichlet energy. In this paper we want to study stationary surfaces   of the functional $E$ under some geometric properties using techniques of differential geometry. Here we follow the seminal paper  \cite{re}. Instead to use the function $u$, it is more suitable to write $E[u]$  in terms of the surface $\Sigma$. If $\nu =(\nu_1,\nu_2,\nu_3)$  is the unit normal of $\Sigma$, then $\nu=(-Du,1)/\sqrt{1+|Du|^2}$. Thus the Dirichlet energy becomes
\begin{equation}\label{fu}
\mathcal{F}(\Sigma)=\int_\Sigma \left(\dfrac{1}{\nu_3}-\nu_3\right)\, d\Sigma,
\end{equation}
 where $d\Sigma$ is the area element of $\Sigma$.  More generally, we can consider   functionals    of type  $\mathcal{F}(X)=\int_\Sigma F(\nu)\, d\Sigma$, where $X\colon\Sigma\to\r^3$ is an immersion of  $\Sigma$ into $\r^3$ and $F\colon U\subset\s^2\to\r^+$ is a positive smooth function on the unit sphere $\s^2$. These functionals are called {\it anisotropic} because  the energy depends on the normal direction $\nu$ of $\Sigma$. Anisotropic energies appear in fluid phenomenon when the surface tensions of interfaces  depend  on $\nu$ \cite{ta}.  Critical points of   $\mathcal{F}$ for compactly supported volume-preserving variations are characterized by the property that the function $\Lambda$ given by 
$$\Lambda:= 2HF-\mbox{div}_\Sigma DF$$
 is constant, where $DF$ is the gradient of $F$ in $\s^2$. The   function $\Lambda$ is called the {\it anisotropic mean curvature} of   $\Sigma$. Thus, critical points of $\mathcal{F}$, or stationary surfaces, are surfaces with constant anisotropic mean curvature (CAMC).
 
Usually, for general anisotropic energies, it is assumed that    $F$ is elliptic  in the sense that the matrix  $D^2F+F\, \mbox{Id}$ is positive definite, where  $D^2F$ is the Hessian of $F$. This allows to define the Wulff shape associated to $\mathcal{F}$ as   the map $\xi\colon\s^2\to\r^3$ given  by $\xi(\nu)=DF(\nu)+F(\nu)\nu$. This map is the parametrization of a convex surface, which it is an ovaloid if   $U=\s^2$. We point out that   the Dirichlet energy    is elliptic and its Wulff shape is  the paraboloid  $z=x^2+y^2$, after translations and rescalings  in $\r^3$. Notice that the integrand in \eqref{fu} is only defined on a hemisphere of $\s^2$.
 
We   ask for those CAMC surfaces of the Dirichlet energy \eqref{fu} with some particular geometric property. A first example are  the rotational surfaces, that is, surfaces that are axially symmetric about an axis of $\r^3$. This reduces the equation $\Lambda=ct.$ in an ordinary differential equation. Exactly, CAMC surfaces of the Dirichlet energy that are   axially symmetric about the $z$-axis are given by 
 $$u(x,y)= c_1\log(\sqrt{x^2+y^2})+\frac{\Lambda}{8}(x^2+y^2)+c_2,\quad c_1,c_2\in\r.$$
When $c_1=0$ we obtain the Wulff shape and when $c_2=0$ we have the anisotropic catenoid of the Dirichlet energy. For the interesting reader in the differential-geometric viewpoint of   critical points of anisotropic energies, we refer that works of Koiso and Palmer  \cite{kp2,kp3,kp5,kp4} and references therein.  See also the recent papers \cite{ba,gmt,gx,jw,ro}. 

Other interesting family of surfaces are the ruled surfaces. These surfaces are generated by moving a straight line in Euclidean space $\r^3$ and they can be parametrized by 
\begin{equation}\label{para1}
X\colon I\times\r\to\r^3,\quad X(s,t)=\alpha(s)+t\beta(s),
\end{equation}
where $\alpha \colon I\to\r^3$ is a regular curve, called the {\it directrix} which we can suppose to be parametrized by arc-length. The curve $\beta$ is a curve on the unit $2$-sphere $\beta\colon I\to\s^2$ and it indicate the direction of the {\it rulings} at each point $\alpha(s)$. Moreover, we can assume $\langle \alpha(s),\beta(s)\rangle=0$ for all $s\in I$. In case that $\beta$ is a constant curve, $\beta(s)=\vec{w}$, the surface is called {\it cylindrical}. 

 In this paper, we give a complete classification of the ruled CAMC surfaces of the Dirichlet energy \eqref{fu}. Ruled surfaces of cylindrical type are studied separately. We prove in Sect. \ref{s3} that any cylindrical CAMC surface is a plane ($\Lambda=0$) or the directrix $\alpha$ is a parabola, that is, the surface is a parabolic cylinder (Thm. \ref{t2}). 
 
 For non-cylindrical CAMC surfaces, we have that  $\{\beta,\beta',\beta\times\beta'\}$ is an orthonormal frame. Since  $\langle \alpha(s),\beta(s)\rangle=0$, there are smooth functions $a,b\colon I\to\r$ such that 
\begin{equation}\label{ab}
\alpha(s)=a(s)\beta'(s)+ b(s)\beta(s)\times\beta'(s).
\end{equation}
The classification of the ruled non-cylindrical CAMC surfaces is  the following. 

\begin{theorem}\label{t1} Let $\Sigma$ be a ruled non-cylindrical surface in $\r^3$ parametrized by \eqref{para1} where $\alpha$ is given by \eqref{ab}. If the anisotropic mean curvature $\Lambda$  of the Dirichlet energy is constant then $\beta$ is a great circle of $\s^2$. Moreover, up to a rotation about the $z$-axis and translations of $\r^3$, the surface $\Sigma$ is one of the following surfaces:
\begin{enumerate}
\item Case $\Lambda=0$.
\begin{enumerate}
\item A plane.
\item A helicoid. The curve  $\beta$ is the horizontal equator $\beta(s)=(\cos s,\sin s,0)$ and
 \begin{equation}\label{sol1}
\begin{split}
a(s)&=c_1\cos s,\quad  c_1\not=0,\\
b(s)&=b_1\tan(s),\quad b_1\not=0.
\end{split}
\end{equation}

\item The curve $\beta$ is the vertical equator $\beta(s)=(\cos s,0,\sin s)$ and 
\begin{equation}\label{sol2}
\begin{split}
a(s)&=c_1\cos(s+c_2),\quad c_1,c_2\in\r, c_1\not=0,\\
b(s)&=b_1\tan(s),\quad b_1\not=0.
\end{split}
\end{equation}
\item The curve $\beta$ is a (non-horizontal, non-vertical) great circle of $\s^2$ and 
\begin{equation}\label{sol3}
\begin{split}
a(s)&=c_1\cos s+c_2\sin s\\
&- \frac{b_1 m}{1-m^2}\cos s \cot^{-1}(\sqrt{1-m^2}\cot s)+\frac{b_1}{m\sqrt{1-m^2}}\sin s,\\
b(s)&=\frac{b_1}{\sqrt{1-m^2}}\tan^{-1}(\sqrt{1-m^2}\tan s).
\end{split}
\end{equation}
 Here $ c_1,c_2,b_1,m \in\r$ with $b_1\not=0$, $m\in (0,1)$.

\end{enumerate}
\item Case $\Lambda\not=0$. Then $\beta$ is the vertical equator $\beta(s)=(\cos s,0,\sin s)$ and
\begin{equation}\label{sol4}
\begin{split}
a(s)&=c_1\cos s+c_2\sin s- \frac{\Lambda b_1^2\cos(2s)}{4\cos s},\quad c_1,c_2\in\r, \\
b(s)&= b_1\tan s,\quad b_1\not=0.
\end{split}
\end{equation} 
\end{enumerate}
\end{theorem}

 The proof of Thm. \ref{t1} will be carried out in Sect. \ref{s4} separating the case $\Lambda=0$ (Subsect. \ref{s41}) from $\Delta\not=0$ (Subsect. \ref{s42}).

\section{Preliminaries}\label{s2}

In this section, we  recall the calculation of the anisotropic mean curvature of a surface $\Sigma$ when the energy functional is of type $\mathcal{F}(X)=\int_\Sigma F(\nu_3)\, d\Sigma$, and next, we particularize to the Dirichlet energy \eqref{fu}. A first observation is that the equation $\Lambda=ct.$ is preserved by translations of $\r^3$ and by dilations (with different constant $\Lambda$), but not by rigid motions in general. This is because the functional $\mathcal{F}$ depends on the unit normal $\nu$. However, the equation $\Lambda=ct.$  is invariant by rotations about an axis parallel to the $z$-axis.

Let $h\colon T\Sigma\times T\Sigma\to\r$ be the second fundamental form of $\Sigma$, $h(v,w)=-\langle d\nu(v),w\rangle$, $v,w\in T\Sigma$. 
The anisotropic mean curvature $\Lambda$ is given by 
\begin{equation}\label{h}
\Lambda=\frac{h(v_1,v_1)}{\mu_1}+\frac{h(v_2,v_2)}{\mu_2},
\end{equation}
where     $\{v_1,v_2\}$ is    an orthonormal frame of $T\Sigma$ and $\mu_1$ and $\mu_2$ are the principal curvatures of the Wulff shape   given by 
$$\frac{1}{\mu_1}=(1-\nu_3^2)F''(\nu_3)+\frac{1}{\mu_2},\quad \frac{1}{\mu_2}=F-\nu_3F'(\nu_3).$$
The principal directions $\{E_1,E_2\}$ of $\mu_1$ and $\mu_2$ are
$$E_1=e_3-\nu_3\nu,\quad E_2=\nu\times E_1,$$
where $e_3=(0,0,1)$. Here it is understood that $\nu_3^2\not=1$ and  because all our results are local, the $\nu_3^2\not=1$ on $\Sigma$.   In the particular case of the Dirichlet energy \eqref{fu}, we have 
$$\frac{1}{\mu_1}=\frac{2}{\nu_3^3},\quad \frac{2}{\mu_2}=\frac{1}{\nu_3}.$$
We use the basis  $\{E_1,E_2\}$ to compute $\Lambda$ in \eqref{h}. This basis is orthogonal but not unitary and $|E_1|^2=|E_2|^2=1-\nu_3^2$. From the expression of $\mu_1$ and $\mu_2$, identity   \eqref{h} is  
\begin{equation}\label{h2}
\Lambda\nu_3^3(1-\nu_3^2)=2\left(h(E_1,E_1)+ h(E_2,E_2)\nu_3^2\right) . 
\end{equation}
 We compute   \eqref{h2} for an arbitrary parametrization $X=X(s,t)$ of $\Sigma$. Since  $\{X_s,X_t\}$ is a basis of $T\Sigma$,   let express $E_1$ and $E_2$ in coordinates with respect to this basis:
\begin{equation}\label{e1e2}
E_1=c_{11} X_s+c_{12}X_t,\quad 
E_2=c_{21}X_s+c_{22}X_t.
\end{equation}
Then
$$h(E_1,E_1)=c_{11}^2h(X_s,X_s)+2 c_{11}c_{12}h(X_s,X_t)+c_{12}^2h(X_t,X_t),$$
$$h(E_2,E_2)=c_{21}^2h(X_s,X_s)+2 c_{21}c_{22}h(X_s,X_t)+c_{22}^2h(X_t,X_t).$$
Let $g_{ij}$ denote the coefficients of the first fundamental form of $X$ and let $\nu=\frac{X_s\times X_t}{|X_s\times X_t|}$. Then 
\begin{equation*}
\begin{split}
h(X_s,X_s)&=-\langle d\nu(X_s),X_s)\rangle=\frac{\mbox{det}(X_s,X_t,X_{ss})}{\sqrt{\mbox{det}(g_{ij})}}:=\frac{h_{11}}{\sqrt{\mbox{det}(g_{ij})}}\\
h(X_s,X_t)&=-\langle d\nu(X_s),X_t)\rangle=\frac{\mbox{det}(X_s,X_t,X_{st})}{\sqrt{\mbox{det}(g_{ij})}}:=\frac{h_{12}}{\sqrt{\mbox{det}(g_{ij})}}\\
h(X_t,X_t)&=-\langle d\nu(X_t),X_t)\rangle=\frac{\mbox{det}(X_s,X_t,X_{tt})}{\sqrt{\mbox{det}(g_{ij})}}:=\frac{h_{22}}{\sqrt{\mbox{det}(g_{ij})}}.
\end{split}
\end{equation*}
Then Eq. \eqref{h2} writes as
\begin{equation}\label{h3}
\begin{split}
\Lambda\nu_3^3(1-\nu_3^2) \sqrt{\mbox{det}(g_{ij})}&=
2 (c_{11}^2h_{11}+2c_{11}\,c_{12}h_{12} +c_{12}^2h_{22})\\
&+2\nu_3^2 ( c_{21}^2h_{11}+2c_{21}\,c_{22}h_{12} +c_{22}^2h_{22}).
\end{split}
\end{equation}

Immediately,    any plane of $\r^3$ is a stationary surface   for $\Lambda=0$ because  the second fundamental vanishes identically.   Notice that if $\Sigma$ is a horizontal plane, then $\nu_3^2=1$ and we cannot use \eqref{h3}.

\section{Stationary cylindrical surfaces } \label{s3}

A cylindrical surface $\Sigma$ is a   ruled surface where   the rulings are all parallel. A parametrization of $\Sigma$ is 
\begin{equation}\label{para2}
X\colon I\times\r\to\r^3,\quad X(s,t)=\alpha(s)+t\vec{w},
\end{equation}
where $\alpha \colon I\to\r^3$ is a   curve  parametrized by arc-length, $|\vec{w}|=1$ and $\alpha$ is contained in a plane orthogonal to $\vec{w}$.  In the following result, we classify all cylindrical CAMC surfaces.   

\begin{theorem} \label{t2}
Let $\Sigma$ be a cylindrical surface. If $\Sigma$ is a stationary surface of the Dirichlet energy \eqref{fu}, then $\alpha$ is a straight-line and $\Sigma$ is a plane ($\Lambda$) or $\alpha$ is a parabola and $\Sigma$ is a parabolic cylinder ($\Lambda\not=0$).
\end{theorem}

\begin{proof}
A first trivial case is when $\Sigma$ is a plane ($\Lambda=0$). From now, we will be discarded this case.   Consider the parametrization \eqref{para2}. Let $\n$ be the unit normal of $\alpha$ defined by $\n=\alpha''/|\alpha''|$, where $\kappa=|\alpha''|$ is the curvature of $\alpha$. Consider the orientation on  $\alpha$ such that $\alpha'\times \vec{w}=\n$. Since $X_s=\alpha'$ and $X_t=\vec{w}$, then $\nu=X_s\times X_t=\n$. Let write $e_3=(0,0,1)$ in coordinates with respect to $\{\alpha',\vec{w},\n\}$, 
$$e_3=e_{11}\alpha'+e_{22} \vec{w}+e_{33}\n.$$
Notice that $e_{33}=\langle e_3,\nu\rangle=\nu_3\not=0$.  
Recall that we discard the case $\nu_3^2= 1$ identically (horizontal plane). 

We compute all terms of \eqref{h3}. It is immediate  
 \begin{equation*}
\begin{split}
E_1&=e_{11}\alpha'+e_{22}\vec{w},\\
E_2&=-e_{22}\alpha'+e_{11}\vec{w}.
\end{split}
\end{equation*}
Since $h_{11}=\kappa$, $h_{12}=h_{22}=0$ and $\nu_3=e_{33}$, equation \eqref{h3} is 
\begin{equation}\label{ec}
\Lambda e_{33}^3(1-e_{33}^2)=2\kappa (e_{11}^2+e_{22}^2e_{33}^2).
\end{equation}
Consider a positively oriented (constant) basis $\mathcal{B}=\{v_1,v_2,v_3\}$ of $\r^3$ such that $v_3=\vec{w}$. Since the curve $\alpha$ is included in the $(v_1,v_2)$-plane and it is parametrized by arc-length, then there is a smooth function $\theta=\theta(s)$ such that 
$$\alpha'=\cos\theta v_1+\sin\theta v_2.$$
Thus 
$$\n=\alpha'\times \vec{w}=\sin\theta v_1-\cos\theta v_2.$$
Moreover, $\kappa=\langle\alpha'',\n\rangle=-\theta'$. Let write $e_3$ in coordinates with respect to $\mathcal{B}$, 
\begin{equation}\label{cij}
e_3=(a_1,a_2,a_3)=(\cos\varphi\cos\phi,\cos\varphi\sin\phi,\sin\varphi)
\end{equation}
for some $\varphi,\phi\in\r$.  Then 
\begin{equation*}
\begin{split}
e_{11}&=a_1\cos\theta+a_2\sin\theta,\\
 e_{22}&=a_3,\\
 e_{33}&=a_1\sin\theta-a_2\cos\theta.
 \end{split}
 \end{equation*}
Thus \eqref{ec} becomes 
\begin{equation}\label{ec2}
\begin{split}&\frac{\Lambda}{2}(a_1\sin\theta-a_2\cos\theta)^3(1-(a_1\sin\theta-a_2\cos\theta)^2)\\
&=-2\theta'\left((a_1\cos\theta+a_2\sin\theta)^2+a_3^2(a_1\sin\theta-a_2\cos\theta)^2\right).
 \end{split}
 \end{equation}
 Taking into account the value of $a_i$ in \eqref{cij}, equation \eqref{ec2} is  
\begin{equation}\label{ec3}
\theta'=-\frac{\Lambda}{2}(\cos\varphi)^3(\cos(\theta(s)+\phi))^3.\end{equation}
 If $\Lambda=0$, then $\kappa=-\theta'=0$ and $\alpha$ is a straight-line and $\Sigma$ is a plane. This case was already considered. If $\Lambda\not=0$,  equation \eqref{ec3} is the expression of  the curvature of a parabola when it is parametrized by arc-length. This proves the result.
 \end{proof}
 
 We give two examples.
 \begin{example}\label{ex1} Suppose that $\alpha$ is contained in a vertical plane which, after a rotation about the $z$-axis, we assume  the $xz$-coordinate plane. Then the vector $\vec{w}=\pm (0,1,0)$. Recall that $\vec{w}$ is the third vector in the basis $\mathcal{B}$. Let $\mathcal{B}=\{(0,0,1),(1,0,0),(0,1,0)\}$. Let choose $\varphi=\phi=0$. Then \eqref{ec3} is 
  $$\kappa=-\frac{\Lambda}{2}(\cos\theta )^3=\frac{\Lambda}{2}(\cos\theta)^3.$$
 If we write the curve $\alpha$ as $\alpha(x)=(x,0,z(x))$, then $\kappa=z''/(1+z'^2)^{3/2}$ and $\sin\theta=1/(1+z'^2)^{3/2}$. Thus \eqref{ec3} is simply $z''=\frac{\Lambda}{2}$. This is according to \eqref{hh} when we assume  functions $z=z(x,y)=z(x)$. The solution is the straight-line $z(x)=z_0$ ($\Lambda=0$) or the parabola $z(x)=\frac{\Lambda}{4}x^2+a x+b$, $a,b\in\r$. See Fig. \ref{fig0}, left.
 \end{example}
 
 \begin{example}\label{ex2} Let 
$$\mathcal{B}=\{(1,0,0),\frac{1}{\sqrt{2}}(0,1,-1),\frac{1}{\sqrt{2}}(0,1,1)\}.$$
Then $(a_1,a_2,a_3)=(0,-\frac{1}{\sqrt{2}},\frac{1}{\sqrt{2}})$. In this case, $\varphi=\pi/4$ and $\phi=-\pi/2$. Then  \eqref{ec2} becomes 
$$\kappa =\frac{\Lambda}{8}(\sin\theta)^3.$$
See Fig. \ref{fig0}, right.
\end{example}

\begin{figure}[hbtp]
\centering
\includegraphics[width=.35\textwidth]{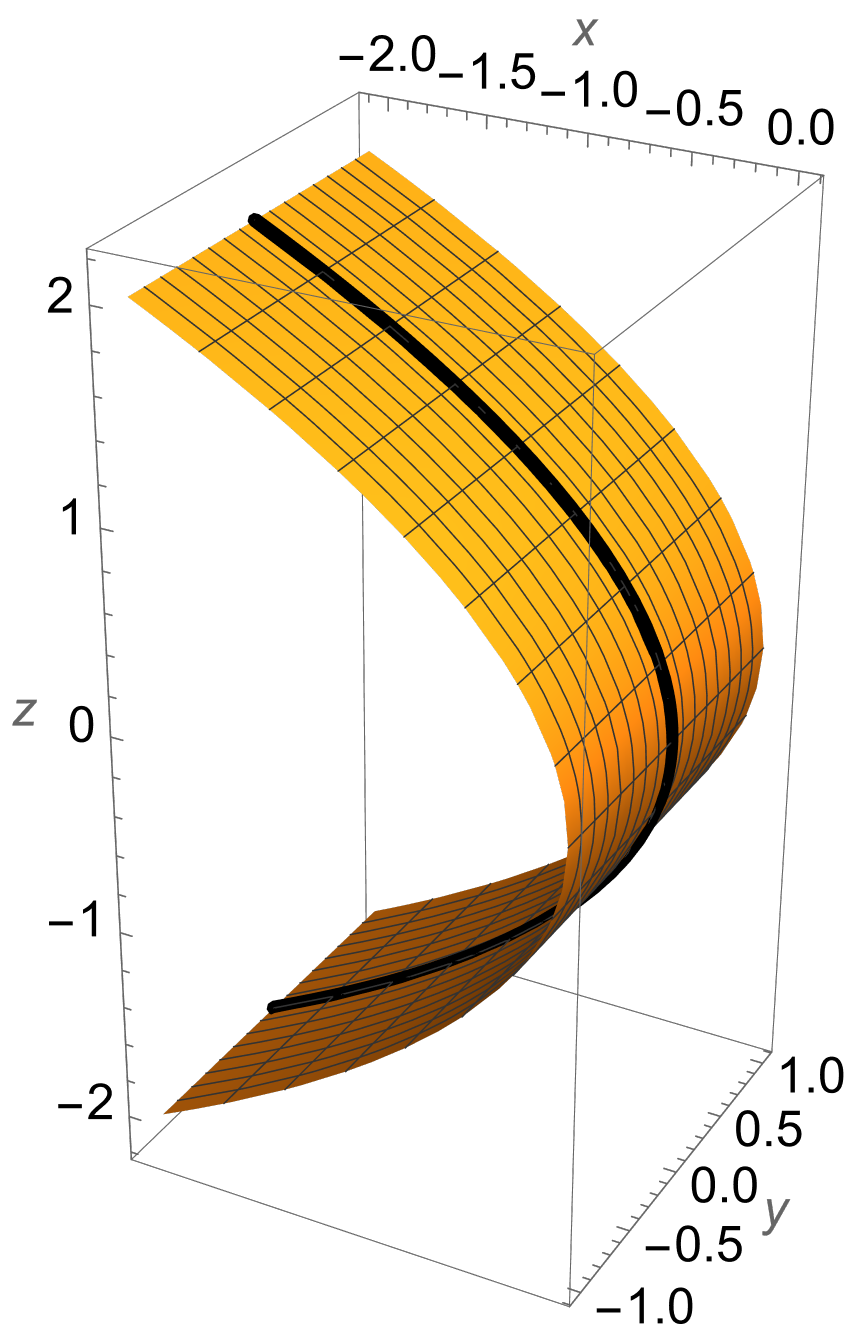}\quad \includegraphics[width=.35\textwidth]{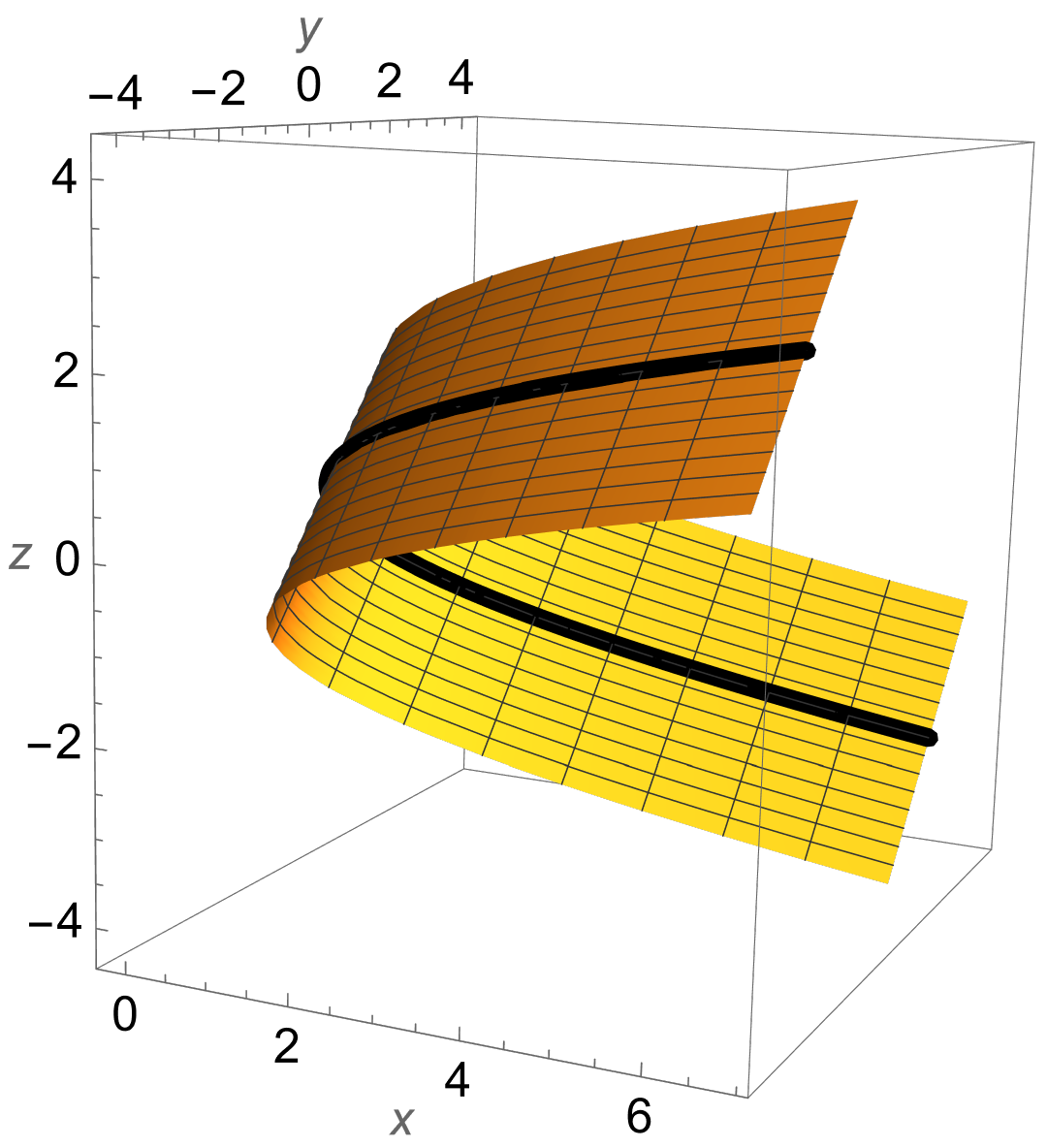}
\caption{Cylindrical CAMC surfaces for $\Lambda=2$. Left: Example \eqref{ex1}. Right: Example \eqref{ex2}. }\label{fig0}
\end{figure}

\section{Non-cylindrical  ruled surfaces } \label{s4}

Suppose that $\Sigma$ is a non-cylindrical ruled surface. Then $\Sigma$ is parametrized by \eqref{para1} where $\beta$ is not a constant curve. Since $\beta$ is a spherical curve, we associate an orthonormal frame $\mathcal{B}=\{v_1,v_2,v_3\}$ given by
\begin{equation*}
\left\{\begin{split}
v_1(s)&=\beta(s),\\
v_2(s)&=\beta'(s),\\
v_3(s)&=\beta(s)\times\beta'(s).
\end{split}
\right.
\end{equation*}
Let $\kappa=\mbox{det}(\beta'',\beta,\beta')$ be the curvature of $\beta$ viewed as a curve of $\s^2$. The corresponding Frenet equations are
\begin{equation*}
\left\{
\begin{split}
v_1'(s)&=v_2(s),\\
v_2'(s)&=-v_1(s)+\kappa(s) v_3(s),\\
v_3'(s)&=-\kappa(s) v_2(s).
\end{split}
\right.
\end{equation*}
Notice that if $\kappa\equiv 0$, then $\beta$ is a geodesic of $\s^2$ and $v_3=\beta\times\beta'$ is constant. In this case,  $\beta$ is a great circle in $\s^2$ which it is the intersection of $\s^2$ with the vector plane orthogonal to $v_3$.

Since $\alpha$ is orthogonal to $\beta=v_1$, the expression of $\alpha$ in coordinates with respect to $\mathcal{B}$ is 
$$\alpha(s)=a(s) v_2(s)+b(s) v_3(s),$$
for smooth functions $a=a(s)$ and $b=b(s)$. In particular, the parametrization of $\Sigma$ writes in coordinates with respect to $\mathcal{B}$ is
\begin{equation}\label{para}
X(s,t)=(t,a(s),b(s)).
\end{equation}

The unit normal of $\Sigma$ is 
$$\nu=\frac{X_s\times X_t}{|X_s\times X_t|}=\frac{\alpha'\times v_1- t v_3}{\sqrt{W}},\quad W:=|\alpha'\times v_1- t v_3|^2.$$

We have $X_{ss}=\alpha''+t\beta''$, $X_{st}=\beta'$ and $X_{tt}=0$. After some calculations, the coefficients $g_{ij}$ of the first fundamental form are 
\begin{equation*}
\begin{split}
g_{11}&=t^2+2(a'-\kappa b)t+a^2+(a'-\kappa b)^2+(b'+\kappa a)^2,\\
g_{12}&=a,\\
g_{22}&=1,\\
\mbox{det}(g_{ij})&=W=t^2+2(a'-\kappa b)t +(a'-\kappa b)^2+(b'+\kappa a)^2.
\end{split}
\end{equation*}
On the other hand, the determinants $h_{ij}$ are
\begin{equation*}
\begin{split}
h_{11}&=-\left((a \kappa+b') (-a''+a (\kappa^2+1)+2 \kappa b'+b \kappa')\right)\\
&-\left(a'-b \kappa+t\right) \left(\kappa (2 a'-b \kappa+t)+a \kappa'+b''\right), \\
h_{12}&=b'+\kappa a,\\
h_{22}&=0\\
\end{split}
\end{equation*}
We write $e_3=(0,0,1)$  in coordinates with respect to $\mathcal{B}$,  
$$e_{3}=(e_{11}(s),e_{22}(s),e_{33}(s)).$$
 Since $\beta$ is not constant,   the coordinates $e_{22}$ and $e_{33}$ cannot be zero simultaneously. The calculations of $c_{ij}$ in \eqref{e1e2}  give
\begin{equation*}
\begin{split}
c_{11}&=\frac{e_{22} \left(a'-b \kappa +t\right)+e_{33} a \kappa +e_{33} b'}{W},\\
c_{12}&=\frac{1}{W}\Big( a \left(e_{22} \left(a'-b \kappa +t\right)+b' (2 e_{11} \kappa +e_{33})\right)+e_{11} (b'^2+(a'-b \kappa +t)^2)\\
&+a^2 \kappa  (e_{11} \kappa +e_{33})\Big),\\
c_{21}&=-\frac{e_{11}}{\sqrt{W}},\\
c_{22}&=\frac{e_{22} \left(a'-b \kappa +t\right)+a (e_{33} \kappa -e_{11})+e_{33} b'}{\sqrt{W}}.
\end{split}
\end{equation*}
Finally, 
\begin{equation}\label{nu3}
\nu_3=\frac{-e_{33} a'+e_{22} a\kappa+e_{22} b'+e_{33} b\kappa-e_{33} t}{\sqrt{W}}.
\end{equation}
Inserting all these computations in \eqref{h3}, we obtain a polynomial equation on $t$ of type
\begin{equation}\label{st}
\sum_{n=0}^5 A_n(s) t^n=0.
\end{equation}
Therefore, all coefficients $A_n$ vanish identically. The explicit computations of $A_n$ can be done after tedious calculations, increasing the degree of difficulty as we go from $n=5$ to $n=0$.   In this paper we have employed Mathematica for these computations \cite{wo}. 

The first coefficient to compute is  $n=5$, obtaining
$$A_5=e_{33}^3(e_{33}^2-1)\Lambda.$$
We separate the cases $\Lambda=0$ (Subsect. \ref{s41}) and $\Lambda\not=0$ (Subsect. \ref{s42}) .

In many of the next computations, we will use the relation $e_{11}^2+e_{22}^2+e_{33}^2=1$. We will also use the expression \eqref{nu3} to assure that $\nu_3\not=0$. 

In all cases, the arguments are the following. First we compute $A_n$ for the biggest $n$. Next, we impose the condition $A_n=00$ and this will give a discussion of cases. In each one of them, we compute $A_{n-1}$ and study equation $A_{n-1}=0$. We repeat the process decreasing the value of $n$  until we finish with the coefficient $A_0$ and the corresponding equation $A_0=0$. In some situations, the discussion of $A_m=0$ for some $m$ will imply that the rest of coefficients $A_{m-1},\ldots, A_0$ are trivially $0$. This means that no more computations can do.

\subsection{Case $\Lambda=0$}\label{s41}

The degree of Eq. \eqref{st} is $4$, where
$$A_4=-\kappa(e_{22}^2+e_{11}^2e_{33}^2).$$
Equation $A_4=0$ gives two cases.
\subsubsection{Case $\kappa=0$. } In particular, $\beta$ is a great circle of $\s^2$. The coefficient $A_3$ is 
$$A_3=(e_{11}^2+e_{22}^2)\left(2e_{11}e_{22}b'+(e_{11}^2-1)b''\right).$$
Equation $A_3=0$ gives two cases.
\begin{enumerate}
\item Case $e_{11}^2+e_{22}^2=0$. Then $e_{33}=\pm 1$ and thus $e_3=\pm v_3$. Without loss of generality, let $e_3=v_3=\beta\times\beta'$. This means that $\beta$ is the horizontal equator of $\s^2$.   Now $A_2$ is trivially $0$ and the last two coefficients are
\begin{equation*}
\begin{split}
A_1&=-b'^2b'',\\
A_0&=b'^2(b'(a+a'')-a'b'').
\end{split}
\end{equation*}
Then $A_1=0$ gives two subcases.
\begin{enumerate}
\item Subcase $b'=0$. Then $b(s)=b_0$. From the parametrization \eqref{para}  of $\Sigma$, we have $\langle X,e_3\rangle=b_0$. This proves that the surface is the horizontal plane of equation $z=b_0$.  This cases was considered because $\nu_3^2\equiv 1$. 
\item Subcase $b''=0$ and $b'\not=0$. Then $b(s)=b_1s+b_0$ with $b_1\not=0$. Then $A_0=0$ gives $a''+a=0$. Coming back, we have $$h(E_1,E_1)=\frac{b_1^3}{W^2}(a+a'')=0,\quad h(E_2,E_2)=0.$$
Since $\{E_1,E_2\}$ is an orthonormal basis, we deduce  that $H=0$ identically in $\Sigma$. Thus $\Sigma$ is a minimal surface. Let us observe that we can also compute 
$h_{11}=b_1(a''-a)$, $h_{12}=b_1$, $h_{22}=0$, together $g_{11}=1$ and $g_{12}=-a$. Since the second fundamental form is not identically $0$,  then the surface is not a plane. By the classification of the ruled minimal surface, we conclude that  $\Sigma$ is the helicoid \cite{ca}. 

We calculate explicitly this helicoid. The solution of $a''+a=0$ is a linear combination of the functions $\cos s$ and $\sin s$. Thus $a(s)=c_1\cos(s+c_2)$ for some $c_1,c_2\in\r$. Since $\beta$ is the horizontal equator, then $\beta(s)=(\cos (s+c_2),\sin (s+c_2),0)$. After a translation in the domain of the variable $s$, we can assume $c_2=0$ and a vertical translation of $\r^3$ allows to choose  $b_0=0$.   This gives \eqref{sol1}. See Fig. \ref{fig1}, left.

\end{enumerate}
\item  Case $2e_{11}e_{22}b'+(e_{11}^2-1)b''=0$ and $e_{11}^2+e_{22}^2\not=0$. Notice that $e_{11}^2\not=1$: otherwise,  $e_3=\beta$ and $\Sigma$ would be a cylindrical surface. Then 
\begin{equation}\label{b2}
b''=2b'\frac{e_{11}e_{22}}{1-e_{11}^2}.
\end{equation}
Now 
$$A_2=(e_{11}^2+e_{22}^2)b'\left( (1-e_{11}^2)(a+a'')+2e_{11}e_{33}b'  \right).$$
Since $e_{11}^2+e_{22}^2\not=0$, we discuss two cases from the equation $A_2=0$. 
\begin{enumerate}
\item Case $b'=0$. Then $b(s)=b_0$ is a constant function. For this value of $b$,  the  coefficients $A_1$ and $A_0$ are trivially $0$. Now $h_{11}=h_{12}=h_{22}=0$ so the second fundamental form is $0$ and  $\Sigma$ is a plane.  
\item Case $b'\not=0$. Then the second parenthesis of $A_2$ vanishes identically,  
\begin{equation}\label{a22}(1-e_{11}^2)(a+a'')+2e_{11}e_{33}b'=0.
\end{equation}
From this identity we get $b'$. We have two subcases according whether $e_{33}$ is $0$ or not. Notice that $e_{11}\not=0$. 
\begin{enumerate}
\item Case $e_{33}=0$. This is equivalent to say that the plane containing the geodesic $\beta$ also contains $e_3$. Consequently, $\beta$ is a vertical equator of $\s^2$.
Since $e_{11}^2+e_{22}^2=1$, 
\begin{equation}\label{b1}
b''=2b'\frac{e_{11}}{e_{22}}.
\end{equation}
Moreover, equation \eqref{a22} writes as 
$$ e_{22}^2 b'(a+a'')=0.$$
From $\nu_3\not=0$ in \eqref{nu3}, we know that $e_{22}b'\not=0$. Thus   $a+a''=0$. With this equation,  $A_1$ and $A_0$  are trivially $0$. We now find the parametrization of the surface. We have
$$e_{11}=\langle e_3,\beta\rangle,\quad e_{22}=\langle e_3,\beta'\rangle=e_{11}'.$$
Then $e_{11}^2+e_{22}^2=1$ writes as $e_{11}^2+e_{11}'^2=1$. By solving this equation, we obtain   $e_{11}=\sin(s+c)$ for some constant $c\in\r$. After a translation in the domain of the variable $s$, let $c=0$.  Thus $e_{22}=\cos s$. After a rotation about the $z$-axis, we can assume that $\beta$ is included in the $xz$-coordinate plane and thus, $\beta$ is parametrized by  
 $\beta(s)=(\cos s,0,\sin s)$. From the expression of $b''$ in \eqref{b1}, we solve obtaining $b(s)=b_1\tan(s) $, with $b_1\not=0$. Equation $a''+a=0$ yields $a(s)=c_1\cos(s+c_2)$ for some constants $c_1,c_2\in\r$.  This gives \eqref{sol2}. See Fig. \ref{fig1}, right.

\item Subcase $e_{33}\not=0$. From \eqref{a22}, we get 
$$b'=-\frac{(1-e_{11}^2)(a+a'')}{2e_{11}e_{33}}.$$
Then  \eqref{b2} gives
$$b''=-\frac{e_{22}}{e_{33}}(a+a'').$$
Now the rest of coefficients $A_1$ and $A_0$ are trivially $0$. 

Since $v_3$ is constant, the third coordinate $e_{33}$ too. Let $e_{33}^2=1-m^2$, for some $m\in (0,1)$ and suppose without loss of generality that $e_{33}>0$, $e_{33}=\sqrt{1-m^2}$. Then $e_{11}^2+e_{11}'^2=m^2$. Up to a translation on the domain of the variable $s$, we have $e_{11}=m\sin (s)$ and $e_{22}=m\cos s$. From the expressions of $b'$ and $b''$, we have
$$\frac{b''}{b'}=2m^2\frac{\sin s\cos s}{1-m^2\sin^2 s}.$$
Then 
$$b'=\frac{b_1}{1-m^2 (\sin s)^2},\quad  b_1\not=0.$$
Since $m^2<1$, the solution of this equation is  
$$b(s)=\frac{b_1}{\sqrt{1-m^2}}\tan^{-1}(\sqrt{1-m^2}\tan s).$$
Now we have 
$$a''+a=-\frac{e33}{e22}b''=-\frac{8 b_1 m \sqrt{1-m^2} \sin (s)}{\left(m^2 \cos (2 s)-m^2+2\right)^2}.$$
The solution of $a$ is given in   \eqref{sol3}. 

\end{enumerate}

\end{enumerate}
\end{enumerate}

\subsubsection{ Case $e_{22}^2+e_{11}^2e_{33}^2=0$ and $\kappa\not=0$.}

Now $e_{22}=0$ and $e_{11}e_{33}=0$. If $e_{33}=0$, then $e_{11}=\pm 1$, that is, $e_3=\pm \beta$ and the surface would be cylindrical. Since this case is not possible, then $e_{11}=0$. Thus $e_3=v_3=\beta\times\beta'$ and the curve is a horizontal great circle. This implies $\kappa=0$, which it is not possible.

\subsection{Case $\Lambda\not=0$}\label{s42}

From $A_5=0$ we have $e_{33}=0$ or $e_{33}^2=1$.
\begin{enumerate}
\item Case $e_{33}=0$. Then $A_4=2e_{22}^2\kappa$. If $e_{22}=0$, then $e_{11}=\pm 1$ and $\beta=\pm e_3$ would be a constant vector, which it is not possible. Thus $\kappa=0$. This means that $\beta$ is a great circle of $\s^2$, which it is contained in a vertical plane because   $e_{33}=0$. Now the coefficient $A_3$ is
$$A_3=2e_{22}(e_{22}b''-2e_{11}b').$$
Equation $A_3=0$ gives
$$b''=2\frac{e_{11}b'}{e_{22}}.$$
Inserting into $A_2$, we arrive to
$$A_2=e_{22}^2 b'(-2a+e_{22}\Lambda b'^2-2a'').$$
Notice that $\nu_3=e_{22}b'/\sqrt{W}$, in particular, $b'\not=0$. Thus $A_2=0$ implies
$$a''=-a+\frac12e_{22}\Lambda b'^2.$$
Now the coefficients $A_1$ and $A_0$ are trivially $0$. Since $e_{22}=e_{11}'$ and $e_{11}^2+e_{22}^2=1$, then we can do a similar discussion as in the case $\Lambda=0$, item (1). We have, after a translation on the domain of the variable $s$, $b(s)=b_1\tan s$, $b_1\not=0$.   Then function $a(s)$ satisfies the differential equation 
$$a''+a=\frac{\Lambda b_1^2}{2(\cos s)^3}.$$
The solution is given in \eqref{sol4}.

\item Case $e_{33}^2=1$. Without loss of generality, we assume $e_{33}=1$. This implies $e_3=v_3=\beta\times\beta'$. Thus $\beta$ is the horizontal  equator of $\s^2$, in particular, $\kappa=0$. Now $A_4$ is trivially $0$ and 
$$A_3=-\Lambda b'^2.$$
Then $b(s)=b_0$ is a constant function. Since $e_3=v_3$, from the expression of the parametrization $X$ in \eqref{para}, we have $\langle X,e_3\rangle=b_0$. This proves that $\Sigma$ is a horizontal plane, which it is not possible because we are assuming $\nu_3^2\not=1$.  
\end{enumerate}

We show some pictures of the examples of Thm. \ref{t1}. In Fig. \ref{fig1}, we consider the case $\Lambda=0$. The left surface is the helicoid, where we take $c_1=b_1=1$ in \eqref{sol1}. The directrix is $\alpha(s)=(-\sin s\cos s,(\cos s)^2, s)$. The right surface corresponds with \eqref{sol2} where the curve $\beta$ is a vertical great circle. Here $c_1=b_1=0$ and $c_2=0$ and the directrix is $\alpha(s)=(-\sin s\cos s,-\tan s,(\cos s)^2)$. 

In Fig. \ref{fig2}, we consider the case $\Lambda=2$ and the solution \eqref{sol4}. In the left surface, $c_1=b_1=1$ and $c=0$. In this case, the directrix is the straight-line $\alpha(s)=\frac12(-\tan s,-2\tan s,1)$. The choice of the constants in the left surface are $c_1=3$, $b_1=1$ and $c=0$.

\begin{figure}[hbtp]
\centering
\includegraphics[width=.45\textwidth]{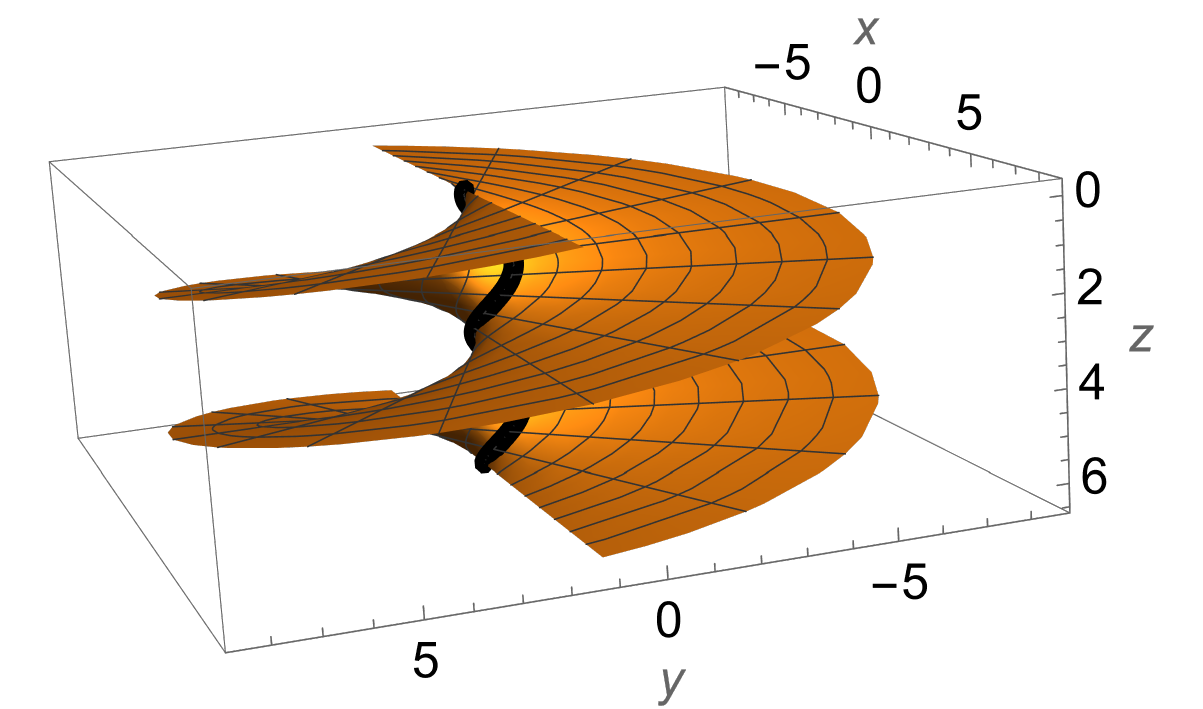}\quad \includegraphics[width=.5\textwidth]{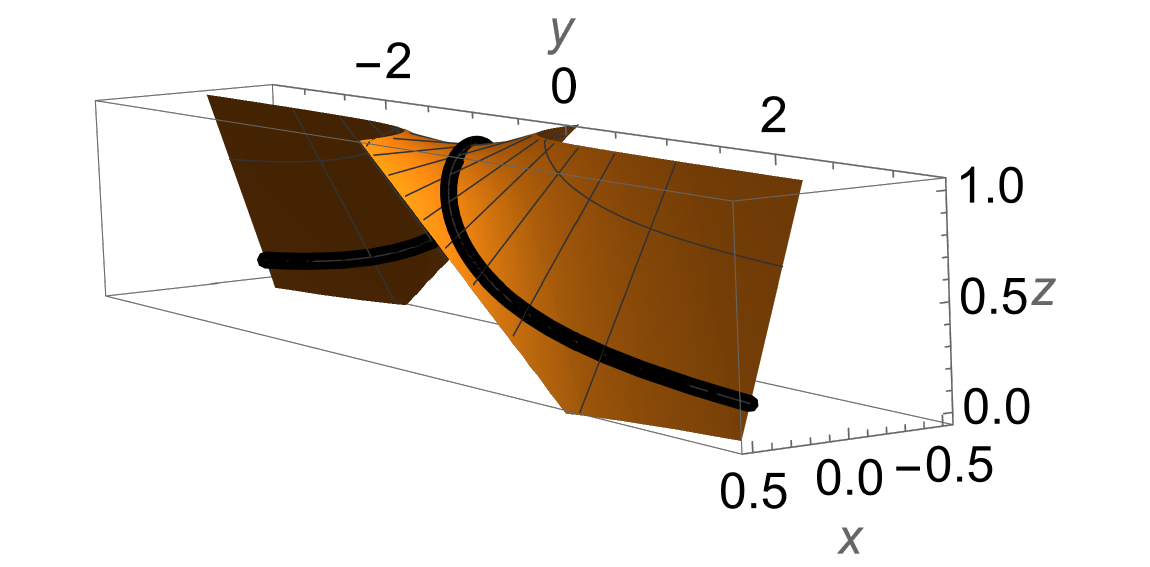}
\caption{Case $\Lambda=0$.  Black line is the directrix $\alpha$. The left surface is the helicoid.}\label{fig1}
\end{figure}

\begin{figure}[hbtp]
\centering
\includegraphics[width=.55\textwidth]{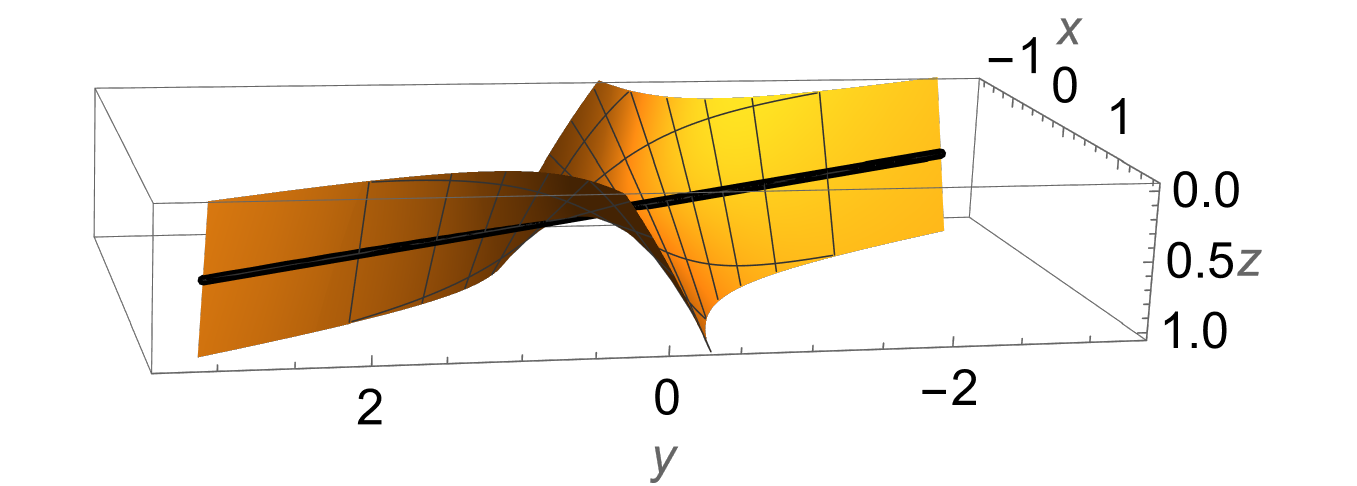}\quad \includegraphics[width=.4\textwidth]{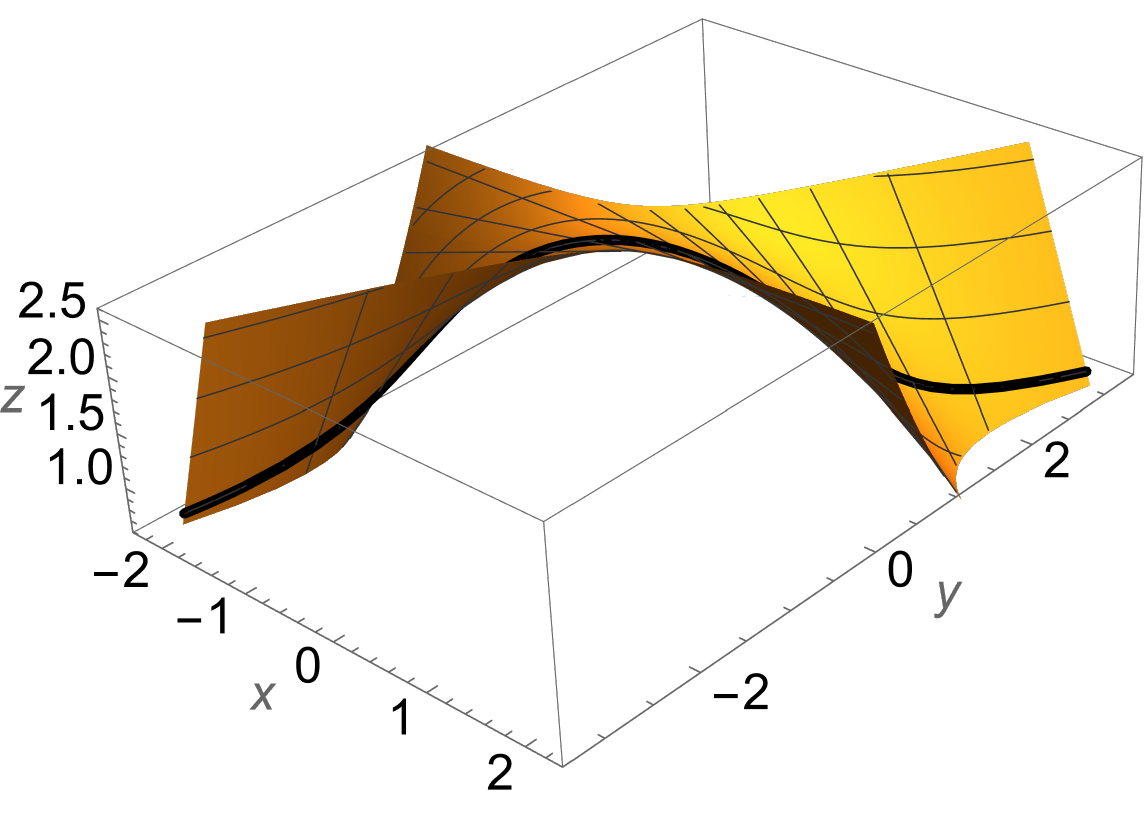}
\caption{Case $\Lambda=2$. Black line is the directrix $\alpha$. }\label{fig2}
\end{figure}

\section*{Acknowledgements}
 The author would like to thank the anonymous referee for insightful suggestions, which helps to improve the exposition of the paper. The author  is a member of the IMAG and of the Research Group ``Problemas variacionales en geometr\'{\i}a'',  Junta de Andaluc\'{\i}a (FQM 325). This research has been partially supported by MINECO/MICINN/FEDER grant no. PID2023-150727NB-I00,  and by the ``Mar\'{\i}a de Maeztu'' Excellence Unit IMAG, reference CEX2020-001105- M, funded by MCINN/AEI/10.13039/ 501100011033/ CEX2020-001105-M.

\end{document}